\documentclass[11pt,a4paper]{amsart}

\usepackage{amsmath,amsthm}
\usepackage{amsfonts}
\usepackage[abbrev,alphabetic]{amsrefs}
\usepackage{verbatim}
\usepackage{amssymb}
\usepackage{txfonts}
\usepackage{amscd}
\usepackage{bbm}
\usepackage[pdfborder={0 0 0 [0 0 ]},bookmarksdepth=3, bookmarksopen=true, unicode]{hyperref}
\usepackage{mathrsfs}
\usepackage{fouriernc}

\usepackage{tikz-cd}

\usepackage[capitalize]{cleveref}

\newtheorem{theorem}{Theorem}
\newtheorem{corollary}[theorem]{Corollary}
\newtheorem{lemma}[theorem]{Lemma}
\newtheorem{proposition}[theorem]{Proposition}

\newcommand{\supp}{\operatorname{supp}}

\newcommand{\RR}{\mathbb{R}}
\newcommand{\ZZ}{\mathbb{Z}}
\newcommand{\A}{\mathcal{A}}

\newcommand{\dH}{d_{\operatorname{Ham}}}
\newcommand{\Ham}{\mathrm{Ham}}

\newcounter{dawidcomments}

\newcounter{piotrcomments}

\author{Dawid Kielak}
\address{Dawid Kielak, Mathematical Institute, University of Oxford, Andrew Wiles Building,
	Radcliffe Observatory Quarter,
	Woodstock Road,
	Oxford
	OX2 6GG,
	United Kingdom}
\email{kielak@maths.ox.ac.uk}
\author{Piotr W. Nowak}
\address{Piotr W. Nowak, Institute of Mathematics of the Polish Academy od Sciences, \'{S}niadeckich 8, 00-656 Warsaw, Poland}
\email{pnowak@impan.pl}

\title{Coboundary expansion and Gromov hyperbolicity}

\begin{document}



\begin{abstract}
We prove that if a compact $n$-manifold admits a sequence of residual covers that form a coboundary expander in dimension $n-2$, then the manifold has Gromov-hyperbolic fundamental group. In particular, residual sequences of covers of non-hyperbolic compact connected irreduc\-ible 3-manifolds are not 1-coboundary expanders.
\end{abstract}

\maketitle

\section{Introduction}

The notion of high-dimensional expanders grew out of the natural need to generalize expansion in graphs to higher dimensions. 
There are a few (inequivalent) notions of such high-dimensional expanders; they include  geometric expanders, spectral expanders, coboundary expanders, cosystolic expanders and topological expanders.  We refer to \cite{Lubotzky2018} as a survey on this topic.
Here we focus on the notion of coboundary expanders, defined independently by Linial--Meshulam \cite{LinialMeshulam2006
}, Gromov \cite{Gromov2010} 
and Doterrer--Kahle \cite{DotterrerKahle12}. 

Coboundary expanders are defined in terms of a spectral gap, with respect to the Hamming norm, of the cohomological Laplacian in appropriate degrees. 
The cohomology here is simplicial cohomology with coefficients in the field $\mathbb{Z}_2$, however we will consider more general coefficients in any non-trivial abelian 
group, as in e.g. \cite{FirstKaufman22}.

Such expansion becomes interesting when one looks at sequences of spaces $X_i$ which grow globally but not locally, and which are expanders with the same associated constant. A prime example of such a situation is given by a family of finite covers of a fixed CW-complex, where the family converges to the universal cover. More explicitly, given a group $G$, a chain of finite-index 
subgroups $(N_i)_i$ is \emph{residual} if $\bigcap_i N_i$ is the trivial group. If $G$ were the fundamental group of a connected compact CW-complex $X$, letting $\widetilde X$ denote its universal cover, we obtain the associated residual chain of coverings $(N_i \backslash \widetilde X)_i$. The key question that we address in this article is: when does such a chain form an $n$-coboundary (or cocycle) expander?

\begin{theorem}
	\label{main}
	Let $X$ be a compact connected triangulated $n$-manifold, possibly with boundary. Suppose that $G = \pi_1(X)$ admits a residual chain $(N_i)_i$ of finite-index 
	 subgroups, and let $X_i = N_i \backslash \widetilde X$ be the corresponding finite covers of $X$. If the sequence $(X_i)_i$ forms an $(n-2)$-cocycle expander with coefficients in some non-trivial abelian group $\A$, then $G$ is Gromov hyperbolic.
\end{theorem}

Since every coboundary expander is a cocycle expander (with the same associated constant), the theorem applies to coboundary expanders as well.
In some special cases we obtain interesting corollaries. The first concerns quotients of the symmetric space associated to $\mathrm{SL}_n(\RR)$, and can be readily generalised to symmetric spaces with spines, or with uniform lattices. 

\begin{corollary}
For every $n\geqslant 3$, the quotients of a closed equivariant regular neighbourhood of the well-rounded spine of the symmetric space $\mathrm{SL}_n(\RR)/\mathrm{SO}_n$ by torsion-free residual chains in $\mathrm{SL}_n(\ZZ)$ are not $m$-coboundary expanders where $m = n^2 - 3 - n(n-1)/2$.
\end{corollary}

The second interesting case is that of $3$-manifolds. Recall that the conjecture of Lubotzky--Sarnak asserts that  
the fundamental group of a finite-volume hyperbolic 3-manifold does not have property $\tau$; that is, for some sequence of finite-index normal subgroups 
intersecting only at the identity the Cayley graphs of the corresponding finite quotients do not form a family of expander graphs. 
The conjecture is true, since finite-volume hyperbolic $3$-manifolds are virtually RFRS, as shown by Agol \cite{Agol2013}, and non-trivial RFRS groups admit finite-index subgroups that are indicable, that is, admit an epimorphism to $\mathbb Z$. The existence of such an epimorphism is an obstruction for property $\tau$. (See also \cite{Longetal2008} for a construction of a family of expanders in this setting.)

Expander graphs are $0$-dimensional coboundary expanders, and so hyperbolicity seems to be an obstruction to the existence of such expanders. In contrast,  here we show that in dimension $1$ hyperbolicity is necessary. 

\begin{corollary}
Residual chains of covers of  compact connected irreducible non-hyperbolic $3$-manifolds are not $1$-coboundary expanders. 
\end{corollary}

\subsection*{Acknowledgements} We would like to thank Alex Lubotzky for insightful comments.

This work has received funding from the European Research Council (ERC) under the European Union's Horizon 2020 research and innovation programme (Grant agreement No. 850930).

PN was supported by  the National Science Center Grant Maestro-13 UMO-2021/42/A/ST1/00306.

\section{Coboundary and cocycle expansion}

Throughout the article, $\A$ is a fixed non-trivial abelian group. 
Let $X$ be a connected CW-complex. We will consider (co)homology of $X$ with coefficients in $\A$. 
By $C_n(X,\A)$ (respectively $C^n(X,\A)$), we denote the set of  cellular $n$-chains (respectively, cellular $n$-cochains) with coefficients in $\A$. 
The set of $n$-cycles will be denoted by $Z_n(X,\A)$, of $n$-cocycles by $Z^n(X,\A)$, and similarly, $B_n(X,\A)$ and  $B^n(X,\A)$ denote sets of boundaries and coboundaries,
respectively. The set of compactly supported $n$-cochains will be denoted by $C^n_\mathrm{comp}(X,\A)$, and similarly, compactly supported 
$n$-cocycles will form $Z^n_{\mathrm{comp}}(X,\A)$. 
By $d$ we will denote the standard differential and by $d^*$ the co\-differential.

For any chain or cochain $c$ in the above setting, the Hamming norm $\Vert c\Vert_\Ham$  is defined to be the cardinality of the support of $c$, 
and the Hamming distance to be $d_\Ham(c,c')=\Vert c-c'\Vert_\Ham$.

A finite CW-complex $X$ is an \emph{$n$-coboundary expander with associated constant $\lambda_n>0$ and coefficients in $\A$} if  for every $n$-cochain $c \in C^n(X; \A)$ the inequality 
\begin{equation}
\| d^*c \|_\Ham \geqslant \lambda_n \dH\left(c, B^n(X;\A)\right)
\end{equation}
holds. This definition has to be modified for $n=0$, when one replaces $B^0(X;\A)$ by $Z^0(X;\A)$. 

Motivated partially by this modification, and partly by the method of our proof, we introduce the following variation.
The complex $X$ is an \emph{$n$-cocycle expander with associated constant $\lambda_n>0$ and coefficients in $\A$} if  for every 
$n$-cochain $c\in C^n(X,\A)$ we have 
\begin{equation}
\| d^*c \|_\Ham \geqslant \lambda_n \dH\left(c, Z^n(X;\A)\right).
\end{equation}
Since the coefficient group $\A$ will be fixed, we will from now on omit it from notation.

Observe that an $n$-coboundary expander for $\lambda_n>0$ satisfies $H^n(X;\A)=0$. Indeed, a non-trivial cohomology class 
is represented by a non-trivial cocycle $c$ satisfying $\Vert d^*c\Vert_{\Ham}=0$ and $d(c, B^n(X,\A))>0$. Thus a couboundary expander is 
a cocycle expander with the same constant. 

A family $\{ X_i \}$ of complexes is a family of $n$-dimensional coboundary, respectively cocycle, \emph{expanders} if there exists a $\lambda_n>0$ such that 
each $X_i$ is an $n$-coboundary expander, respectively $n$-cocycle expander, with  constant $\lambda_n$. 

\section{Locality and lifting of expansion}

We first introduce a useful generalization of the path metric on a graph. Given a CW-complex $X$, we let $X^{(n)}$ denote the set of its $n$-cells.
We will say that a subset $A\subseteq X^{(n)}$  is \emph{coboundary connected} if for any two $n$-cells $\sigma, \tau \in A$ there exits a family of $n$-cocycles 
$\lbrace\tau_0,\dots, \tau_k\rbrace$,  each supported on a single $n$-cell in $A$,  such that
\[\supp (d^* \tau_i)\cap \supp (d^*\tau_{i+1}) \neq \emptyset,\]
and with $\{\tau\}=  \supp \tau_0$ and $\{\sigma\} = \supp \tau_k$.
Being coboundary connected is an equivalence relation on elements of $A$. Equivalence classes of this relation will be called \emph{coboundary-connected components} of $A$.

For two $n$-cells $\tau$ and $\sigma$ in $X^{(n)}$, the minimal $k$ for a chain as above (if it exists) will be 
referred to as the \emph{coboundary distance}. It is easy to see that the coboundary distance is a metric taking values in $\mathbb{N}\cup\{0, \infty\}$
and it reduces to the path metric on the vertices of a graph when $n=0$.

\subsection{Locality of cocycle expansion}
We will now show that cocycle expansion, and in particular coboundary expansion, is local.  
\begin{lemma}
Let $X$ be a compact CW-complex. Consider a cochain $c\colon X^{(n)}\to \A$ and let $z$ be an $n$-cocycle such that 
\[
\dH(c, Z^n(X)) = \dH(c,z).
\]
Then for every non-empty coboundary-connected component $\zeta$ of $\supp z$ we have 
that $\zeta$ and $\supp c$ are not disjoint. 
\end{lemma}

\begin{proof}
Assume the contrary, that is, that there exists $\zeta$, a non-empty co\-boundary-connected component of $\supp z$, such that $\zeta$ and $\supp  c$ are disjoint.
Let $z_0 = z|_\zeta$ and $z' = z - z_0$. Since $\zeta$ is a coboundary-connected component, we have 
\[d^*z_0 = 0 = d^*z'.\]
Moreover, since $\zeta \cap \supp c = \emptyset$, the functions $z'$ and $z$ agree on $\supp c$. 
They obviously also agree on $X^{(n)} \smallsetminus \zeta$. 

As the Hamming distance $\dH(c,z)$ counts the cardinality of the support of the difference $c-z$, and since $\zeta \cap \supp c = \emptyset$,
we obtain that 
$$\dH(c,z')< \dH(c,z)=\dH(z,Z^n(X)),$$
which is a contradiction.
\end{proof}

As an immediate consequence we obtain the following.

\begin{corollary}[Locality of cocycle expansion]
	\label{locality}
Let $X$ be a compact CW-complex, and let $c\colon X^{(n)}\to \A$ be a cochain. If  $z$ is a cocycle satisfying 
\[\dH(c,z)=\dH(c, Z^{n}(X)),\]
then $\supp z$ lies in the ${\dH(c,Z^{n}(X)) }$-neighbourhood of $\supp c$ with respect to the coboundary distance.
\end{corollary}

\subsection{Lifting}

The fact that the cocycle and coboundary expansion implies local estimates on diameters of certain chains as in the previous lemmas, allows us to lift 
the expansion in a residual family of coverings. 

\begin{proposition}[Lifting of cocycle expansion]
	\label{lifting}
Let $X$ be a connected compact CW-complex, and let $(X_i)_i$ be a residual sequence of coverings of $X$. If every $X_i$ is an $n$-cocycle expander with constant $\lambda_n > 0$, then
 for every compactly supported cochain $c \colon  \widetilde{X}^{(n)}\to \A$ we have
$$\Vert d^*c \Vert_{\Ham}  \geqslant \lambda_n {\dH(c,Z_\mathrm{comp}^{n}(\widetilde X)) }.$$
\end{proposition}

\begin{proof}
Let $c\colon \widetilde{X}^{(n)}\to \A$ be finitely supported and let $R = \| d^*c \Vert_\Ham$. Since the sequence $(X_i)_i$ is residual, there exists an $i$ such that 
the quotient map $\rho \colon \widetilde X \to X_i$ takes the $(\frac R {\lambda_n} + 1)$-neighbourhood (with respect to the coboundary distance) of 
$\supp c$ identically to its image.
The cochain $c$ now induces a cochain $c' \colon {X_i}^{(n)} \to \A$ via the restriction of $\rho$ to this neighbourhood in the obvious way. In particular, 
\[
\| d^*c' \|_\Ham = \| d^*c \|_\Ham = R,\]
and consequently
\[
\lambda_n \dH\left(c',Z^{n}(X_i)\right) \leqslant R.
\] 
By \cref{locality}, there exists a cocycle $z' \in Z^{n}(X_i)$ such that
\[
\dH\left(c',z'\right) \leqslant \frac R {\lambda_n}
\]
and $\supp z'$ is contained in the $\frac R {\lambda_n}$-neighbourhood of $\supp c'$. 
However, this neighbourhood is homeomorphic to that of $\supp c$ via $\rho$, and hence we may lift $z'$ to a cocycle $z \in Z^n_{\mathrm{comp}}(\widetilde X)$. Clearly, 
\[
\dH(c,z) = \dH(c',z'),
\] 
and so
\[
\lambda_n \dH\left( c,z \right) \leqslant R  =  \| d^*c \Vert_\Ham
\]
as claimed.
\end{proof}

\section{Homological Dehn functions and hyperbolicity}

We are now going to focus on the situation described in \cref{main}. Hence, we will have a triangulated compact connected $n$-manifold $X$ with $G = \pi_1(X)$.

We are going to show that $G$ is Gromov-hyperbolic by showing that its homological Dehn function over $\A$  in dimension $1$ is linear; that is, 
our task is to find a constant $\kappa > 0$ such that every simplicial $1$-cycle $p \in Z_1(\widetilde X,\A)$ 
is the boundary of a simplicial $2$-chain $q \in C_2(\widetilde X,\A)$ such that
\[
\Vert q\Vert_{\Ham} \leqslant \kappa \Vert p \Vert_{\Ham}.
\]
That this is sufficient
follows from the work of Gersten \cite{Gersten1996}*{Theorem 5.2} for $\A = \ZZ$; for other non-trivial  abelian groups one can use the proof of \cite{KielakKropholler2021}*{Proposition 4.1}.

Observe that, to stay true to Gersten's theorem, we should be working with the universal covering of a classifying space of $G$ rather than $\widetilde X$, which might not be one, since $X$ is not necessarily aspherical. However,
$\widetilde X$ is simply connected, and this is sufficient since we are looking at isoperimetric inequalities in dimension $1$; see the proof of \cite{KielakKropholler2021}*{Proposition 4.1} for details.
 
Recall that $X$ is  triangulated. When considering cochains, we will always use simplicial cochains with respect to the triangulation of $X$ we are given. When considering chains, we will look at the dual of the triangulation, which is a polyhedral structure, and hence in particular a CW-structure. All our chains will be cellular with respect to this structure.
Throughout, all of our chains and cochains will have coefficients in $\A$.

\begin{proof}[Proof of Theorem \ref{main}]
Let $p\in Z_1(\widetilde{X})$ be given. 
Since $\widetilde X$ is simply connected, there exists a $2$-chain $r\in C_2(\widetilde{X})$ with $dr = p$. 

Poincar\'e duality applied to $\widetilde X$ gives us a compactly supported $(n-1)$-cocycle $z\in Z^{n-1}_{\mathrm{comp}}(\widetilde{X})$ dual to $p$, 
and a compactly supported $(n-2)$-cochain $c$ dual to $r$ such that $d^*c = z$. Note that these cochains are also naturally combinatorial with respect to the  triangulation  of $X$, as explained above. Crucially for us, with respect to this combinatorial structure, Poincar\'e duality does not change the Hamming norm, and so 
\[
\Vert r \Vert_{\Ham} = \Vert c \Vert_{\Ham} \text{\ \ \ \ and\ \ \ \ } \Vert p \Vert_{\Ham} = \Vert z \Vert_{\Ham}.
\]

Recall that we are given a residual family
   of finite covers $X_i$ of $\widetilde{X}/G$ forming an $(n-2)$-cocycle expander with constant $\lambda_{n-2}>0$. 
\cref{lifting} yields
\[
{\Vert d^*c\Vert_{\Ham} } \geqslant {\lambda_{n-2}} {\dH(c,Z^{n-2}_\mathrm{comp}(\widetilde X)) }.
\]

Let $\kappa = 1/ {\lambda_{n-2}}$ and let $z' \in Z_\mathrm{comp}^{n-2}(\widetilde X)$ be a compactly supported cocycle such that 
\[
\dH(c,z') =  \dH(c,Z^{n-2}_\mathrm{comp}(\widetilde X)). 
 \]
  Rearranging the above equation yields
\[
\|c-z'\|_\Ham = \dH(c,z') \leqslant \kappa \| z \|_\Ham.
\]
Note that $d^*(c-z') = d^*c = z$.

Applying Poincar\'e duality again we obtain a $2$-chain $q\in C_2(\widetilde{X})$ which is dual to the compactly supported cochain 
$c-z'\in C^{n-2}_{\mathrm{comp}}(\widetilde{X})$,  satisfying $dq = p$ and 
\[
\|q\|_\Ham \leqslant \kappa \| p \|_\Ham .
\]
This proves that $G$ admits a linear homological isoperimetric inequality in dimension $1$, and is therefore Gromov hyperbolic. 
\end{proof}

We remark that while our main result is formulated for dimension $n-2$, the above argument works equally well in other dimensions, as long as $\widetilde X$ is acyclic in the appropriate dimension. 
However, only in dimension 1 the linear isoperimetric inequality has a meaningful geometric interpretation. It would be extremely 
interesting to  characterize  linear isoperimetric inequalities in dimensions other than 1 geometrically.

When $G = \pi_1 (X)$ is infinite, the universal covering $\widetilde X$ is unbounded and hence does not satisfy a linear isoperimetric inequality in dimension $0$. Indeed, for every $\kappa > 0$ there exist points in $\widetilde X$ with distance greater than $\frac 1 {\kappa}$, and hence a suitable $0$-cycle supported on these points will have Hamming norm $2$, but will only be the boundary of a $1$-chain of Hamming norm greater than $\frac 1 {\kappa}$. This implies that residual chains of finite-index covers of a compact connected triangulated $n$-manifold with infinite fundamental group are never $n-1$-cocycle expanders.

\bibliography{bibliography}

\end{document}